\documentclass[12pt]{amsart}
\usepackage[margin=1in]{geometry}
\usepackage[utf8]{inputenc}
\usepackage{amssymb,amsmath,amsthm} 
\usepackage{amsopn}
\usepackage{mathtools} 
\usepackage{xspace}
\usepackage{mathrsfs}
\usepackage{graphicx}
\usepackage{verbatim}
\usepackage{hyperref}
\usepackage{wasysym}
\usepackage{enumitem}
\usepackage{tikz}
\usetikzlibrary{calc,arrows.meta,positioning}
\usepackage{capt-of}
\usepackage{cleveref}
\usepackage[font=small]{caption}
\usepackage{float}

\theoremstyle{plain}
\newtheorem{theorem}{Theorem}
\newtheorem{corollary}[theorem]{Corollary}
\newtheorem{lemma}[theorem]{Lemma}
\newtheorem{proposition}[theorem]{Proposition}

\theoremstyle{definition}
\newtheorem{remark}[theorem]{Remark}

\newtheorem{question}[theorem]{Question}
\numberwithin{equation}{section}
\numberwithin{theorem}{section}

\theoremstyle{definition}

\newtheorem{definitionx}[theorem]{Definition}
\newenvironment{definition}
{\pushQED{\qed}\definitionx}
{\popQED\enddefinitionx}

\newtheorem{examplex}[theorem]{Example}
\newenvironment{example}
{\pushQED{\qed}\examplex}
{\popQED\endexamplex}

\newcommand{\B}[1]{\mathbb #1}

\renewcommand{\P}{\mathbb{P}}

\DeclareMathOperator{\rank}{rank}

\DeclareMathOperator{\gcc}{gcc}

\DeclareMathOperator{\tc}{tc}

\newcommand{\comp}{\mathsf{c}}

\definecolor{DarkGreen}{rgb}{0,0.65,0}

\definecolor{NiceBlue}{rgb}{0.2,0.2,0.75}
\newcommand{\struc}[1]{{\color{NiceBlue} #1}}

\title{Multiple typical ranks in matrix completion}

\author{Mareike Dressler}
\address{Mareike Dressler, School of Mathematics and Statistics, University of New South Wales, Sydney, NSW 2052, Australia. Previously:	Department of Mathematics, 
	University of California, San Diego, 
	La Jolla, CA 92093, 
	USA.\medskip}
\email{m.dressler@unsw.edu.au}

\author{Robert Krone}
\address{Robert Krone, Department of Mathematics,
	University of California, Davis,       
	One Shields Avenue                        
	Davis, CA 95616,                          
	USA\medskip}
\email{rkrone@math.ucdavis.edu}


\begin{document}
	
\begin{abstract}
	Low-rank matrix completion addresses the problem of completing a matrix from a certain set of generic specified entries. Over the complex numbers a matrix with a given entry pattern can be uniquely completed to a specific rank, called the generic completion rank.  Completions over the reals may generically have multiple completion ranks, called typical ranks. We demonstrate techniques for proving that many sets of specified entries have only one typical rank, and show other families with two typical ranks, specifically focusing on entry sets represented by circulant graphs. This generalizes the results of Bernstein, Blekherman, and Sinn. In particular, we provide a complete characterization of the set of unspecified entries of an $n\times n$ matrix such that $n-1$ is a typical rank and fully determine the typical ranks of an $n\times n$ matrix with unspecified diagonal for $n<9$. Moreover, we study the asymptotic behavior of typical ranks and present results regarding unique matrix completions.
\end{abstract}

\maketitle

\section{Introduction}

The problem of recovering a low-rank matrix from a subset of the entries with specified values, also known as \struc{\emph{low-rank matrix completion}}, is ubiquitous in diverse areas and many applications like collaborative filtering (e.g., the famous ``Netflix problem'') \cite{Goldberg1992UsingCF},  computer vision \cite{Chen2004RecoveringTM,Tomasi2004ShapeAM}, control \cite{Mesbahi:Papa}, machine learning \cite{Amit:Fink:Serebro:Ullman}, and signal processing \cite{Weng:Wang}.

We can view the low-rank matrix completion problem for a matrix $A$ as the optimization problem of finding the lowest rank matrix $X$ which coincides with $A$ in all specified (known, sampled) entries. This is a non-convex optimization problem and well-known to be NP-hard in general. However, there exist several heuristic methods that solve this problem approximately or exactly by considering convex relaxations of the original problem. One of the most popular approaches in this manner is proposed by Fazel in \cite{Fazel:Thesis} who suggests approximating the rank condition by using the nuclear norm (sum of its singular values) and to formulate the resulting problem as a semidefinite program that can be solved efficiently. Even more is true: under a certain genericity assumption on an $n\times n$ matrix $A$ of rank $r$ then, with high probability, one can use the aforementioned approach to uniquely recover the matrix $A$ from only $\Theta(n^{1.25}r \log n)$ entries chosen uniformly at random \cite{Candes:Recht, Candes:Tao}. In many applications, the matrices under consideration meet these assumptions, but in case they are not valid one has to rely on different approaches. Moreover, solving the relaxed optimization problem is computationally very expensive if the problem size becomes large. 

Building on ideas from rigidity theory Kir\'{a}ly, Theran, and Tomioka \cite{Kiraly:Theran:Tomioka} developed in 2015 a new approach using tools from algebraic geometry, graph theory, and matroid theory. This novel direction entailed several new works in recent years \cite{BBS,Bernstein:Blekherman:Lee,Tsakiris}. 

Inspired by the results in \cite{BBS}, we study here the \emph{exact} low-rank matrix completion problem. In more detail, we are concerned with the following task: Given a \emph{partially filled} matrix $A$ with \emph{fixed} positions of specified and unspecified entries, we want to find the unspecified entries of $A$ such that the completed matrix has the lowest possible rank, called the \textit{completion rank} of $A$. We typically focus on \emph{generic} input data for $A$, that is, small perturbations of the entries of $A$ do not change its completion rank. It is shown in \cite{Kiraly:Theran:Tomioka} that, outside of a lower dimensional set, all partially filled matrices can be completed to the same minimal rank over the complex numbers. This rank is called the \struc{\emph{generic completion rank}} of the given pattern. However, when restricting the entries of the matrix and its completion to real numbers, this no longer holds true. Depending on the choice of the specified entries, the partially filled matrix $A$ may be completable to different ranks. To address this phenomenon, these ranks are referred to as \struc{\emph{typical ranks}} when the set of matrices minimally completable to these ranks forms a full-dimensional subset of the vector space of the specified entry values (i.e., the space of partial fillings). This is in analogy to the study of generic and typical ranks of tensors \cite{Blekherman:Teitler, Kahle:Kubjas:Kummer:Rosen, Bernardi:Blekherman:Ottaviani:RealTypicalRanks, Landsberg}.

\subsection*{Contributions}
In this article we are particularly interested in studying typical ranks of a given pattern of specified and unspecified entries. Our leading goal is to gain a deeper understanding of the sets of specified and unspecified entries respectively exhibiting multiple typical ranks. 

In this light, our first main contribution is a complete characterization of the set of unspecified entries of an $n\times n$ matrix such that $n-1$ is a typical rank, see \Cref{thm:Characterization1isTypicalCR}. 

The given entry pattern of an $n\times m$ matrix translates nicely into the notion of a bipartite graph with parts of size $n$ and $m$, see \Cref{ex:graphANDmatrix}.
In order to approach our goal we lay our focus on the study of matrices with unspecified set given by certain bipartite circulant graphs, denoted \struc{$G(n,k)$}, see \Cref{def:Gnl}.  These graphs have shown to be particularly challenging to complete over the real numbers to particular ranks.  The smallest graph with multiple typical ranks is $G(4,1)$, and in \cite{BBS} it was posited that other graphs in the family might be examples with larger numbers of typical ranks. Our second major contribution is the full determination of the typical ranks for unspecified entry set $G(n,1)$ for $n<9$, see \Cref{4diag}, \Cref{5diag}, and \Cref{5to8diag}. In the course of these results we successfully answer several questions raised in \cite{BBS}. Moreover, we present results that bound the growth of the typical ranks of $G(n,1)$ (\Cref{1circulant}) and of $G(n,k)$ (\Cref{kcirculant}) as $n$ grows. Further outcomes on deciding the typical ranks of explicit graphs are given in Propositions \ref{G64e} and \ref{G'52}. In sum these results shed light onto the situation of multiple typical ranks and our search for examples of when they occur. However, we are still far away from completely grasping this phenomenon. One major question that remains is whether a set of entries with three or more typical ranks is possible. 

Our remaining results concern generically unique completability of a matrix, see, e.g., Lemma~\ref{grow}. 

Furthermore we pose several open problems and questions for the important future study of typical ranks.

\bigskip
The paper is structured as follows. In Section \ref{Sec:Prelim}, we provide the necessary background in algebraic geometry and for matrix completions. Section~\ref{Sec:MainTheorem} studies generically unique completability of a matrix, presents our results on the typical ranks corresponding to the unspecified set of entries being $G(n,n-1)$, and states our main theorem. In Section~\ref{Sec:more circulant graph results}, we study the generically real completion of $G(n,n-1)$. We identify a certain subset of $G(n,n-k)$ that yields to further outcomes for bipartite circulant graphs. We conclude with a brief discussion in Section~\ref{Sec:Discussion}.

\section{Preliminaries}
\label{Sec:Prelim}

Let $\struc{K}$ be a field, and $\struc{T}$ a finite indexing set.  Fix a $|T|$-dimensional affine space with coordinates $\{\struc{x_i}\}_{i \in T}$, denoted $\struc{K^T}$.  Let $\struc{V}$ be an irreducible affine algebraic variety in $K^T$.  For $\struc{S} \subseteq T$, the map $\struc{\pi_S}\colon K^T \to K^S$ will denote the projection map onto the coordinates indexed by $S$.  Throughout we will use the notation $\struc{[n]} \coloneqq \{1,\ldots,n\}$.

\begin{definition}
 For $p \in K^S$, a {\struc{\em completion}} of $p$ to $V$ is an element of $\pi_S^{-1}(p) \cap V$.  The set $S$ is {\struc{\em completable}} to $V$ if every point $p \in K^S$ has a completion, and {\struc{\em generically completable}} to $V$ if a generic point $p \in K^S$ has a completion.  $S$ is {\struc{\em generically uniquely completable}} to $V$ if a generic point $p \in K^S$ has a unique completion, i.e., a generic fiber of $\pi_S|_V$ consists of one reduced point.
\end{definition}

The ring of polynomials with coefficients in $K$ and indeterminants indexed by $T$ will be denoted $\struc{K[T]}$. Let $\struc{\B I(V)} \subseteq K[T]$ be the ideal of polynomials that vanish on $V$.  Let $\struc{\sigma}\colon K[T] \to K[T]/\B I(V) = \struc{K[V]}$ be the quotient map.  Then $S$ is generically completable to $V$ if and only if the set $\{\sigma(x_i) \mid i \in S\}$ is algebraically independent in $K[V]$.  By a dimension count it follows that if $S$ is generically completable to $V$, then $|S| \leq \dim V$.  The set of completions for a generic $p \in K^S$ has dimension $\dim V - |S|$.

If $S$ is generically uniquely completable to $V$ then $\{\sigma(x_i) \mid i \in S\}$ is a basis for the \emph{algebraic matroid} of $K[V]$.  The converse is not true since $\sigma(S)$ being a basis for $K[V]$ only implies that $\pi_S|_V$ is generically finite-to-one. 
 
\begin{proposition}\label{ratmap}
  If $S$ is generically uniquely completable to $V$ then each coordinate of the completion of $p \in K^S$ to $V$ can be expressed as a rational function of the coordinates of $p$.
\end{proposition}
\begin{proof}
  If $S$ is generically uniquely completable to $V$ then the projection map $\pi_S|_V\colon V \to K^S$ is injective, and its image is Zariski open in $K^S$.  Therefore $\pi_S|_V$ is a birational equivalence, so there is a rational map $K^S \to V$.
\end{proof}

 We are particularly interested in the case that $K^T = \struc{\B C^{[n]\times [m]}}$, the space of $n\times m$ matrices with complex entries, and $V$ is the set of $n\times m$ matrices with rank at most $\struc{r}$ for some chosen nonnegative integer $r$.  When $S \subseteq [n] \times [m]$ is completable to $V$, we say $S$ is completable to rank $r$. We call a point $A \in \B C^S$ a \struc{\emph{partially filled}} matrix, or a \struc{\emph{partial filling}} of $S$.  The set $S$ is referred to as the \struc{{\em specified entries}} of the matrix because a choice of point $A \in \B C^S$ specifies the values of these entries.  The remaining entries $\struc{U} = T \setminus S$ are the \struc{{\em unspecified entries}}.

 \begin{definition}
  The \struc{{\em generic completion rank}} of $S \subseteq [n]\times [m]$ is the lowest rank to which $S$ is generically completable.  Similarly the \struc{{\em generic completion corank}} of $S$, is $\min\{n,m\}$ minus the generic completion rank.
 \end{definition}
 
 It is known that in general the lowest rank to which $S$ is generically completable over the real numbers may be higher than the generic completion rank.  In other words, there may be generic points in $\B R^S$ for which the only completions to the generic completion rank are not real.  In fact, it may be that there are multiple open subsets of $\B R^S$ with different minimum ranks to which $S$ can be completed.  
 
 \begin{definition}
  The \struc{{\em typical ranks}} of $S \subseteq [n]\times [m]$ are all integers $r$ for which there is an Euclidean open subset of $\B R^S$ such that $r$ is the lowest rank to which $S$ can be completed over $\B R$. In the same way we define the \struc{{\em typical coranks}} of $S$ by $\min\{n,m\}$ minus the typical ranks. 
 \end{definition}
 We point out that Zariski dense subsets give an equivalent definition as the sets involved are full-dimensional semialgebraic subsets in $\B R^S$.
 Note that both the generic completion rank and the typical ranks depend not just on a set $S$ of pairs of integers, but also on the matrix size $T = [n] \times [m]$.
 
 The following result was observed in \cite[Theorem~1.1]{Bernardi:Blekherman:Ottaviani:RealTypicalRanks} in context of ranks in projective geometry and in \cite[Proposition 2.8]{BBS} from the matrix completion point of view.
 \begin{theorem}\label{thm:trHierarchy}
 The typical ranks of $S \subseteq [n]\times [m]$ are consecutive integers, and the smallest is equal to the generic completion rank of $S$.
 \end{theorem}
 
 \begin{example}
  Let $S \subseteq [4]\times [4]$ be the set of the 12 off-diagonal entries.  
  The generic completion rank of $S$ is 2 and generally there are 4 completions over the complex numbers.  The typical ranks of $S$ are 2 and 3.  Bernstein, Blekherman, and Sinn \cite[Example~3.1]{BBS} give an example of a partially filled matrix $A \in \B R^S$ for which all 4 completions to rank 2 are not real, and for which this also holds in an open ball around $A$.
 \end{example}

We may interpret the entry set of an $n\times m$ matrix as the complete bipartite graph $\struc{K_{n,m}} = [n] \times [m]$.   A subset of the entries is a bipartite graph $G$ on the disjoint union $[n]\cup [m]$. Hence, we also identify the entry set of a matrix with its corresponding graph. The graph representation is convenient because generic completion and typical ranks are invariant under row and column permutations and transposition of the matrix.  Therefore each case can be represented by an unlabeled bipartite graph.  

In what follows, we consider specific bipartite circulant graphs. Bipartite circulant graphs are bipartite graphs whose biadjacency matrix is a circulant matrix, see \cite{Meyer}. We focus on the following subset of those graphs.

\begin{definition}\label{def:Gnl}
Let $G$ be a bipartite graph on the disjoint union $[n]\cup [n]$, then denote by $\struc{G(n,k)}$ the graph, where the edges of $G$ are $(i,i),(i,i+1),\ldots,(i,i+k-1)$ for each $i\in [n]$, with second index taken modulo $n$.
\end{definition}

Note, that this definition differs from the one given in \cite[Section 5.3]{BBS} for $G(n,k)$, but only by a cyclic column (or row) permutation.  It is the same unlabeled bipartite graph.  Our $G(n,k)$ is the complement of their $G(n,n-k)$.

\begin{example}\label{ex:graphANDmatrix}
For a set of specified entries $G$ it is often easier to describe the set of unspecified entries, which is its \emph{bipartite complement} $\struc{G^\comp} = K_{m,n}\setminus G$.  Note that $G(n,k)^\comp = G(n,n-k)$ up to a permutation of the vertex labels.  Consider the unspecified entry set $G(n,1)$. Its edges are the diagonal entries, $(i,i)$ for $i\in [n]$.  The following are three types of diagrams representing this set:  a bipartite graph; named specified entries and question mark unspecified entries in a matrix; and black and white dots representing specified and unspecified entries respectively.

\vspace{2.5em}
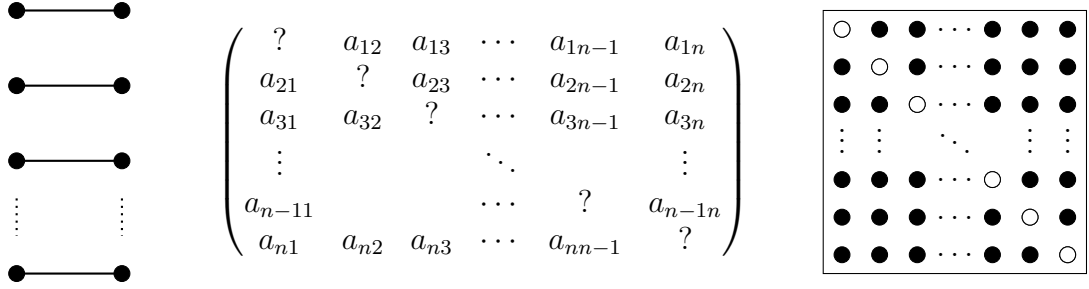
\begin{minipage}{0.22\textwidth}
\centering
\begin{tikzpicture}
\draw[fill=black] (0,0) circle (3pt);
\draw[fill=black] (0,-1) circle (3pt);
\draw[fill=black] (0,-2) circle (3pt);
\draw[dotted,thick] (0,-2.5) -- (0,-3);
\draw[fill=black] (0,-3.5) circle (3pt);
\draw[fill=black] (1.4,0) circle (3pt);
\draw[fill=black] (1.4,-1) circle (3pt);
\draw[fill=black] (1.4,-2) circle (3pt);
\draw[dotted,thick] (1.4,-2.5) -- (1.4,-3);
\draw[fill=black] (1.4,-3.5) circle (3pt);

\draw[thick] (0,0) -- (1.4,0);
\draw[thick] (0,-1) -- (1.4,-1);
\draw[thick] (0,-2) -- (1.4,-2);
\draw[thick] (0,-3.5) -- (1.4,-3.5);
\end{tikzpicture}
\end{minipage}
\begin{minipage}{0.48\textwidth}
$\begin{pmatrix}
? & a_{12} & a_{13} & \cdots & a_{1n-1} & a_{1n}\\
a_{21} & ? & a_{23} & \cdots & a_{2n-1}& a_{2n}\\
a_{31} & a_{32} & ? & \cdots & a_{3n-1}& a_{3n}\\
\vdots & & & \ddots  & & \vdots\\
a_{n-11}& & & \cdots  & ? & a_{n-1n}\\
a_{n1} & a_{n2} &a_{n3} &\cdots & a_{nn-1}& ?
\end{pmatrix}$
\end{minipage}
\begin{minipage}{0.22\textwidth}
	\begin{tikzpicture}
	\foreach \x in {0,1,2,4,5,6}
	\foreach \y in {0,1,2,4,5,6}
	\draw[fill=black] (\x*0.5,\y*-0.5) circle (3pt);
	\foreach \x in {0,...,2}
	\draw[fill=white] (\x*0.5,\x*-0.5) circle (3pt);
	\foreach \x in {4,...,6}
	\draw[fill=white] (\x*0.5,\x*-0.5) circle (3pt);
	\foreach \x in {3}
	\foreach \y in {0,1,2,4,5,6}
	\path (\x*0.5,\y*-0.5) -- node[auto=false]{$\ldots$} (\x*0.5,\y*-0.5);
	\foreach \x in {0,1,5,6}
	\foreach \y in {3}
	\path (\x*0.5,\y*-0.46) -- node[auto=false]{$\vdots$} (\x*0.5,\y*-0.46);
	\foreach \x in {3}
	\foreach \y in {3}
	\path (\x*0.5,\y*-0.47) -- node[auto=false]{$\ddots$} (\x*0.5,\y*-0.47);
	\draw (-.25,.25) rectangle (3.25,-3.25);
	\end{tikzpicture}
\end{minipage}
\vspace{2em}

The above bipartite graph represents any partially filled square matrix with one unspecified entry in each row and different columns, hence, up to interchanging rows, it is a matrix with all entries filled except for its diagonal ones. In what follows, we represent the entries of matrices by black and white dots corresponding to specified and unspecified ones, respectively.  In future diagrams we sometimes omit the black dots, hence the specified entries.

 \Cref{figG73} gives another example with $n = 7$ and $k = 3$.
\end{example}

\section{Unique typical ranks by unique completability}
 \label{Sec:MainTheorem}

In this section we study when a set of unspecified entries admits a generically unique completion to certain minor conditions.  This tool is used to bound the number of typical ranks, particularly in the case of a unique typical rank.  We answer questions about $G(n,1)^\comp$ for both specific and general $n$, leading to a characterization of partially filled square matrices with typical rank $n-1$.
\smallskip

First we discuss the case when fixing $S$, thus its partial filling $A$, and adding a specific number of fully specified rows and columns to $A$. 
\begin{lemma}\label{grow}
 Let $U \subseteq [n]\times [m]$ and define $S = ([n]\times[m]) \setminus U$ and $S' = ([n+k]\times[m+k]) \setminus U$.  Then the generic completion corank and typical coranks of $S$ and $S'$ are equal.  Additionally, $S$ is generically uniquely completable to corank $r$ if and only if $S'$ is.
\end{lemma}
\begin{proof}
 Let $A$ and $A'$ be generic partial fillings of $S$ and $S'$ respectively that agree on $S$ (either over the complex or real numbers).  Then $A'$ can be assigned the following block structure
  \[ A' = \begin{bmatrix} A & B \\ C & D \end{bmatrix}. \]
 Using the Schur complement, $A'/D = A - BD^{-1}C$, we know
 \[ \rank(A') = \rank(D) + \rank(A'/D). \]
 $D$ is fully specified and generic, thus it has rank $k$.  Since $BD^{-1}C$ is totally specified and generic, $A'/D$ is also a generic partial filling of $S$.  Therefore the minimum ranks of completions of $A$ and $A'/D$ are equal.  The completion of $A$ to corank $r$ is unique if and only if the completion of $A'/D$ is unique.
\end{proof}
 
\Cref{grow} tells us that for a fixed set of unspecified entries, $U$, the generic completion corank and the set of typical coranks of $([n]\times [n+k]) \setminus U$ do not depend on $n$.  In particular for square matrices, these coranks depend only on $U$.
\begin{definition}
 For a finite set of unspecified entries $U \subseteq \B Z_{\geq 0} \times \B Z_{\geq 0}$, we will denote the generic completion corank and set of typical coranks of $([n]\times [n]) \setminus U$ by $\struc{\gcc(U^\comp)}$ and $\struc{\tc(U^\comp)}$ respectively. Clearly, $U^\comp$  coincides with the set of specified entries. 
\end{definition}
 
A dimension count shows that
 \[ \gcc(U^\comp) \leq \sqrt{|U|}. \]
By \Cref{thm:trHierarchy} it follows, that the typical coranks are also at most $\sqrt{|U|}$.
Similarly for a non-square matrix of size $(n+k) \times n$, the generic completion corank or a typical corank $r$ of $U^\comp$ satisfies the bound
 \[ r(r+k) \leq |U|. \]

\begin{example}
A simple case where equality is met for $(n+k)\times n$ matrices is when $U$ is a complete bipartite graph, $K_{r+k,r} = [r+k] \times [r]$.  Here $U^\comp$ is generically uniquely completable to rank $r$.  However, the $r$ unspecified entries need not be in the same columns for each of the $r+k$ rows.  The left side of \Cref{fig62} shows the case where $U = K_{4,4}$ in a $6\times 6$ matrix.   
\end{example}
The previous example can be seen as a consequence of the following.
\begin{proposition}[Corollary 4.2 of \cite{BBS}]\label{kcore}
 If $A$ is a generic complex or real matrix of rank $r$, then the partially filled matrix $B$ obtained by adding a partially filled column (or row) to $A$ with at most $r$ generically filled entries has a complex or real (resp.) completion to generic rank $r$.  If the number of filled entries is exactly $r$, then the completion is unique.
\end{proposition}

In Corollary~4.2 \cite{BBS} the authors do not address the uniqueness of the completion. However, this fact follows easily by basic arguments from linear algebra, see also the example below. 

\begin{example}\label{codimc}
 For $(n+k) \times n$ matrices with $r \leq n$ let the unspecified set $U$ be a bipartite graph with $r+k$ vertices of degree $r$ in the first bipartition, and the others of degree $0$.  Then $U^\comp$ is generically uniquely completable to corank $r$.
 \end{example}
For convenience of the reader we include a short proof of Example~\ref{codimc}.
 \begin{proof}
 For $S = [n+k]\times[n] \setminus U$, let $A \in \B C^S$ be a generic partially filled matrix.  There are $n-r$ rows of $A$ with all entries specified.  These rows span a generic dimension $n-r$ subspace $\Lambda \subseteq \B C^n$. The remaining $k+r$ rows have $r$ unspecified entries arranged among $r$ distinct columns. That is, each other row has $n-r$ specified entries.  By genericity, using linear algebra there is a unique completion of the row to $\Lambda$.
 \end{proof}

Now we state another, less obvious, example for square matrices whose entry set is generically uniquely completable.
 
 \begin{corollary}\label{diagstrip}
 For $k \leq n$, the set
  \[ \{(i,j) \in [n]\times [n] \mid n-k+1 < i + j \leq n+k+1 \} \]
 of specified entries is generically uniquely completable to rank $k$. 
\end{corollary}

\begin{proof}
 Define
  \[ \struc{S_{n,k}} = \{(i,j) \in [n]\times [n] \mid n-k+1 < i + j \leq n+k+1 \}. \]
 Fix the value of $k$ and proceed by induction on $n$.  For $n = k$, there are no unspecified entries so $S$ is generically rank $k$.  For $n > k$ let $A \in \B C^{S_{n,k}}$ be a partially filled matrix, and let $B$ be the submatrix of the last $n-1$ rows and last $n-1$ columns.  The set of specified entries of $B$ are $S_{n-1,k}$ after reversing the order of the rows and columns.  By the induction hypothesis, $B$ has a unique completion to rank $k$.  The last $n-1$ entries of the first column of $A$ have $k$ specified entries so they can be completed to be in the column span of the completion of $B$.  With those entries filled, the first row of $A$ has $k$ specified entries, and can be completed to be the row span of the partial filling of $A$.  See the right side of \Cref{fig62} for the case of $n=6$ and $k=2$.
\end{proof}
 
\begin{figure}
\begin{minipage}{.35\textwidth}
 \centering
\begin{tikzpicture}
\foreach \x in {0,...,5}
  \foreach \y in {0,...,5}
    \draw[fill=black] (\x*0.5,\y*-0.5) circle (3pt);
\foreach \x in {0,...,3}
  \foreach \y in {0,...,3}
    \draw[fill=white] (\x*0.5,\y*-0.5) circle (3pt);
\draw (1.75,.25) rectangle (2.75,-2.75);
\draw (1.25,.25) rectangle (2.75,-2.75);
\draw (.75,.25) rectangle (2.75,-2.75);
\draw (.25,.25) rectangle (2.75,-2.75);
\draw (-.25,.25) rectangle (2.75,-2.75);
\end{tikzpicture}
\end{minipage}
\begin{minipage}{.35\textwidth}
\centering
\begin{tikzpicture}
\foreach \x in {0,...,5}
  \foreach \y in {0,...,5}
    \draw[fill=black] (\x*0.5,\y*-0.5) circle (3pt);
\draw[fill=white] (0,0) circle (3pt);
\draw[fill=white] (0,-.5) circle (3pt);
\draw[fill=white] (0,-1) circle (3pt);
\draw[fill=white] (0,-1.5) circle (3pt);

\draw[fill=white] (.5,0) circle (3pt);
\draw[fill=white] (.5,-.5) circle (3pt);
\draw[fill=white] (.5,-1) circle (3pt);

\draw[fill=white] (1,0) circle (3pt);
\draw[fill=white] (1,-.5) circle (3pt);

\draw[fill=white] (1.5,0) circle (3pt);
\draw[fill=white] (1.5,-2.5) circle (3pt);

\draw[fill=white] (2,-2) circle (3pt);
\draw[fill=white] (2,-2.5) circle (3pt);

\draw[fill=white] (2.5,-1.5) circle (3pt);
\draw[fill=white] (2.5,-2) circle (3pt);
\draw[fill=white] (2.5,-2.5) circle (3pt);
\draw (.75,-.75) rectangle (1.75,-1.75);
\draw (.75,-.75) rectangle (2.25,-1.75);
\draw (.75,-.75) rectangle (2.25,-2.25);
\draw (.25,-.75) rectangle (2.25,-2.25);
\draw (.25,-.25) rectangle (2.25,-2.25);
\draw (.25,-.25) rectangle (2.75,-2.25);
\draw (.25,-.25) rectangle (2.75,-2.75);
\draw (-.25,-.25) rectangle (2.75,-2.75);
\draw (-.25,.25) rectangle (2.75,-2.75);
\end{tikzpicture}
\end{minipage}
\caption{Black and white dots represent specified and unspecified entries respectively in a $6\times 6$ matrix.  Each case starts with a generic rank 2 matrix and a partially filled row or column is added one at a time with exactly 2 specified entries.  Both sets of 16 unspecified entries are generically uniquely completable to corank $4 = \sqrt{16}$.}
\label{fig62}
\end{figure}
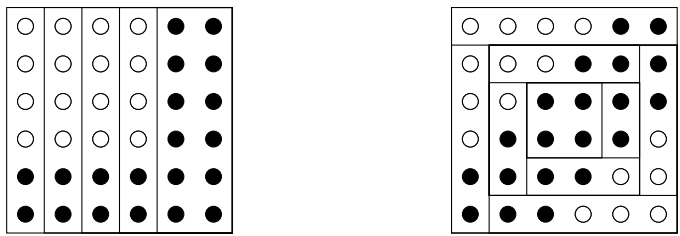

Both \Cref{codimc} and \Cref{diagstrip} can be seen partially as consequences of Corollary~4.5 of \cite{BBS} as well.  For this, recall the notion of a $k$-core from graph theory.

\begin{definition}
The {\em \struc{$k$-core}} of a graph $G$ is the maximal connected subgraph of $G$ in which all vertices have degree at least $k$. Equivalently, it is one of the connected components of the subgraph of $G$ formed by repeatedly deleting all vertices of degree less than $k$.
\end{definition}
In \Cref{codimc}, the set of specified entries has empty $(n-r+1)$-core, so the maximum typical rank is less than $n-r+1$, meaning it has a real completion to corank $r$.  In \Cref{diagstrip}, the $(n-k+1)$-core is empty.  However, in each case we also show that the completion is unique.

\medskip
The following is a corollary of \Cref{codimc}.

 \begin{corollary}
  $S \subseteq [n+k]\times [n]$ has $0$ as a typical corank if and only if the entries of $S$ include an $n\times n$ submatrix.
 \end{corollary}
 \begin{proof}
 	Let $A \in \B R^S$ be a generic real partially filled matrix that includes a completely specified $n\times n$ submatrix.  This submatrix has full rank, so any completion of $A$ has rank $n$. The other direction follows from \Cref{kcore}.
\end{proof}

The next result shows an explicit bipartite graph exhibiting two typical coranks.
 
 \begin{proposition}\label{4diag}
  $\tc(G(4,1)^\comp) = \{1,2\}$.
 \end{proposition}
 \begin{proof}
  Example 3.1 of \cite{BBS} proves that for $4\times4$ matrices the typical coranks of $U^\comp$ are 1 and 2.  This and \Cref{grow} imply the statement for general square matrices.
 \end{proof}

By algebraic arguments we can deduce that an entry set $S$ is generically real completable to a complex variety $V$, if this is true for a subvariety of $V$.
  \begin{proposition}\label{uniquetoreal}
  If $V$ is a complex variety with a subvariety $W$ cut out by real polynomials such that $S$ is generically uniquely completable to $W$, then $S$ is generically completable to $V$ over the reals.
 \end{proposition}
 \begin{proof}
  Over $\B C$, the Zariski closure of $\B R^S$ is $\B C^S$, so there is a generic point $p \in \B C^S$ that is real.  The vanishing ideal of $\pi_S^{-1}(p) \cap W$ is $(S - p) + \B I(W)$ which has real generators.  For a point $q \in \pi_S^{-1}(p) \cap W$, the complex conjugate of $q$ is also in $\pi_S^{-1}(p) \cap W$, but $q$ is unique so it must be real.  Therefore $S$ is generically completable to $W$ over the reals.  Any real completion of $S$ to $W$ is also a completion to $V$.
  \end{proof}

Note that in particular, every variety defined by the vanishing of minors is cut out by real polynomials, including matrix rank varieties.

\medskip
For an $n\times m$ matrix $\struc{A}$ and sets $\struc{P} \subseteq [n]$ and $\struc{Q} \subseteq [m]$, let $\struc{A^Q}$ denote the submatrix of $A$ with column set $Q$, and $\struc{A_P}$ the submatrix of $A$ with row set $P$.  The submatrix $\struc{A^Q_P}$ has column set $Q$ and row set $P$.

\begin{lemma}\label{rcols}
 Suppose that $S \subseteq [n]\times[m]$ is generically uniquely completable to rank $r$.  Let $A\in \B C^S$ be a generic partial filling of $S$, and $A'$ its unique completion to rank $r$.  Then $A'$ is a generic rank-$r$ matrix.  In particular, ${A'}^{[r]}$ is a generic $n\times r$ matrix.
\end{lemma}

\begin{proof}
 By \Cref{ratmap}, the map from the variety of rank-$r$ matrices to $\B C^S$ is a birational equivalence, thus the preimage of $A$ is a generic point on the variety.
\end{proof}

We now turn our attention towards typical coranks of $G(n,1)^\comp$. 
In \cite[Proposition~5.7]{BBS} it is proved that $\gcc(G(n,1)^\comp) = \lfloor \sqrt{n} \rfloor$, which is the bound given by the dimension count.  This is therefore the maximum typical corank.  In Question~5.8 of the original version the authors ask whether 1 is always the minimum typical corank of $G(n,1)^\comp$, which is true for $n \leq 4$.  The subsequent proposition gives a negative answer for $n \geq 5$.

\begin{proposition}\label{5diag}
 $\tc(G(5,1)^\comp) = \{2\}$.
\end{proposition}
\begin{proof}
Let $V$ be the variety of $5\times 5$ matrices with rank 3, and let $W$ be the subvariety of $V$ of matrices in which the first 3 columns are linearly dependent.

With $S = G(5,1)^\comp$, let $A \in \B C^S$ be a generic partially filled $5\times 5$ matrix.

The submatrix $A^{\{1,2,3\}}$ has unspecified entries $(1,1),(2,2),(3,3)$ so by \Cref{kcore} it has a unique completion to rank $2$.  Let $A'$ be the partially filled matrix obtained by this completion of $A^{\{1,2,3\}}$.  Since ${A'}^{\{1,2,3\}}$ is a generic rank 2 matrix, the second and third columns form a generic $5\times 2$ matrix.  By \Cref{kcore}, ${A'}^{\{2,3,4,5\}}$ can be uniquely completed to rank 3.  The result is the unique completion of $A$ to $W$.

 By \Cref{uniquetoreal}, $G(5,1)^\comp$ is generically real completable to corank 2.  By a dimension count, $G(5,1)^\comp$ cannot generically be completed to corank greater than 2.
\end{proof}

\begin{remark}
 The principle we want to emphasize in the proof of \Cref{5diag} is that a generic real partially filled matrix for $G(5,1)^\comp$ has an infinite number (dimension 1) of complex completions to corank 2, but it is difficult to show that any of them are real.  By constraining the problem by adding the extra dependence condition on the first three columns, the problem becomes easier.  There is a unique complex completion that is clearly real.
\end{remark}

\Cref{5diag} implies that 1 is never a typical corank of $U^\comp = G(n,1)^\comp$ for $n \geq 5$ since adding additional diagonal entries to $U$ makes it strictly easier to complete to corank 2 over the reals.  In fact we can now say the following.
\begin{corollary}\label{5to8diag}
$\tc(G(n,1)^\comp) = \{2\}$ for $5 \leq n < 9$.
\end{corollary}
\begin{proof}
This follows directly from the above considerations and because the generic completion corank is 2 in these cases.
\end{proof}

The generic completion corank of $G(9,1)^\comp$ goes up to 3.  We do not know the full set of typical coranks for $n\geq 9$. Therefore it would be interesting to address the following question. 

\begin{question}
Is $2$ a typical corank of $G(9,1)^\comp$?
\end{question}
 
 Experimental evidence suggests that the number of generic completions of $G(9,1)^\comp$ to corank 3 is probably 44 (compared to 4 generic completions of $G(4,1)^\comp$ to corank 2). The above question is equivalent to whether there are generic real partial fillings such that all completions are non-real.

\medskip
For square matrices, we fully characterize the sets $U$ with $1 \in \tc(U^\comp)$.
 \begin{theorem}\label{thm:Characterization1isTypicalCR}
  Let $U$ be a set of unspecified entries.  Then $1$ is a typical corank of $U^\comp$ if and only if $U$ is in one of the following cases:
  \begin{enumerate}
   \item\label{case:block} $|U| \geq 1$ and all entries are in the union of a single column and a single row.
   \item\label{case:3} $U = G(3,1)$.
   \item\label{case:4} $U = G(4,1)$.
  \end{enumerate}
 \end{theorem}
 \begin{proof}
  If $U$ is in case (\ref{case:block}) then for $n\times n$ matrices, there is an $(n-1)\times(n-1)$ fully specified submatrix. So $U^\comp$ cannot be completed to corank greater than $1$.  Since there is an unspecified entry, $U^\comp$ can always be completed to corank $1$ over the reals.  Therefore $\tc(U^\comp) = \{1\}$.
  
  If $U$ is in case (\ref{case:3}), $U^\comp$ cannot generically be completed to corank $2$ by a dimension count, so $\tc(U^\comp) = \{1\}$.
  
  If $U$ is in case (\ref{case:4}), then $\tc(U^\comp) = \{1,2\}$ by \Cref{4diag}.
  
  To prove that $1$ is not a typical corank of $U^\comp$ in all other cases, we proceed according to $|U|$.  If $|U|$ is 1, 2 or 3, then $U$ is in either case (\ref{case:block}) or (\ref{case:3}).  For $|U|=4$, consider the possibilities for how many unspecified entries share a single row or column.  If there are three or more entries in one row or one column, then $U$ is in case (\ref{case:block}).  If each entry is in a unique row and column then $|U|$ is in case (\ref{case:4}).  Therefore we can assume that there is a row or column with exactly 2 entries.  Without loss of generality, take these entries to be $(1,1)$ and $(2,1)$.  If the remaining two entries are in the same row, $U$ is again in case (\ref{case:block}), so assume they are $(a,b),(c,d)$ with $b,d \neq 1$ and $a \neq c$.
  
  Let $A \in \B R^{U^\comp}$ be an $n\times n$ generic partial filling.  The submatrix $A^{\{1\}^c}$ has 2 unspecified entries, $(a,b)$ and $(c,d)$ in different rows, so by \Cref{kcore}, it is real completable to rank $n-2$.  Fixing these two entries to get $A'$, the $n\times(n-1)$ submatrix ${A'}^{\{n\}^c}$ has 2 unspecified entries $(1,1)$ and $(2,1)$, and the rows excluding the first two are generically independent.  By \Cref{kcore}, it can also be completed over the reals to rank $n-2$.  By genericity, columns 1 and $n$ both are linearly dependent on the others, so the whole filled matrix has corank $2$.
  
  Next consider the case of $|U| \geq 5$.  If all entries of $U$ are in distinct rows and columns, \Cref{5diag} shows that 1 is not a typical corank. Otherwise without loss of generality there is a row or column with at least 2 entries of $U$, and without loss of generality we can take them to be $(1,1)$ and $(2,1)$. Then either $U$ is in case (\ref{case:block}) or there are two entries of $U$ not in column 1 and in different rows.  In this case, as shown above, $U^\comp$ is generically real completable to corank 2.
 \end{proof}

 Note that $U = G(4,1)$ is the only case where 1 is a typical corank but not the generic completion corank.  It is the only set of unspecified entries $U$ with $\gcc(U^\comp) \leq 2$ that has more than one typical corank. From the study above, we pose the following important questions for future research.
 
 \begin{question}
   Given $U$ with $\gcc(U^\comp) = 3$, when is $2$ a typical corank?
 \end{question}
 \begin{question}
   Can we characterize the sets $U$ such that $2 \in \tc(U^\comp)$?
 \end{question}

\section{Typical rank asymptotics}
\label{Sec:more circulant graph results}

In this section we dive deeper into the study of real completions of circulant graphs. We specifically analyze generic real completability of $G(n,1)^\comp$ and a certain subset of $G(n,k)^\comp$, which in turn yields answers for the general case $G(n,k)^\comp$ as well. This study leads to bounds on the growth rate of the generic ranks of $G(n,1)$ and $G(n,k)$ for increasing $n$, and hence, also for its typical ranks. 
\smallskip

We start by considering $G(n,1)^\comp$. 
 
\begin{proposition}\label{1circulant}
 For each $r \geq 0$, $G(n,1)^\comp$ is generically real completable to corank $r$ for all $n \geq (4^r-1)/3$.
\end{proposition}

\begin{proof}
 Let $\struc{c_r} = (4^r-1)/3$.  For $n \geq c_r$, $G(c_r,1) \subseteq G(n,1)$.  If $G(c_r,1)^\comp$ is generically real completable to corank $r$, then $G(n,1)^\comp$ is generically real completable to corank $r$.  So it suffices to prove the case of $n = c_r$.  Proceed by induction on $r$.  For $r=1$, a $1\times 1$ matrix with unspecified entry set, $G(c_1,1)^\comp$ is real completable to corank 1 by filling with a 0.
 
 For $r\geq1$, $c_{r+1}$ satisfies the recurrence relation $c_{r+1} = 4c_r + 1$.  Let $A$ be a generic partially filled $c_{r+1}\times c_{r+1}$ matrix with unspecified entries $G(c_{r+1},1)$. It follows from \Cref{codimc}, that the submatrix of the last $2c_r+1$ rows, $A_{\{2c_r+1,\ldots,4c_r+1\}}$ is uniquely completable to rank $2c_r$.  Let $A'$ be the partial filling of $A$ with these entries filled.  Since the filling of  $A_{\{2c_r+1,\ldots,4c_r+1\}}$ is unique, it is a generic rank $2c_r$ matrix, and its first $2c_r$ rows are generic.  So ${A'}_{\{1,\ldots,4c_r\}}$ is a generic partial filling of a $4c_r\times 4c_r+1$ matrix.  By \Cref{grow}, ${A'}_{\{1,\ldots,4c_r\}}$ is real completable to rank $4c_r - r$ if and only if the Schur complement $B = {A'}_{\{1,\ldots,4c_r\}}/{A'}^{\{2c_r+1,\ldots,4c_r+1\}}_{\{2c_r+1,\ldots,4c_r\}}$ is real completable to rank $2c_r - r$.
 
 The submatrix $B^{\{c_r+1,\ldots,2c_r+1\}}$ is uniquely completable to rank $c_r$ by \Cref{kcore}, and the first $c_r$ columns are generic.  Let $B'$ be the partial completion of $B$ with these entries filled.  Again by \Cref{grow}, ${B'}^{\{1,\ldots,2c_r\}}$ is real completable to rank $2c_r-r$ if and only if the Schur complement $C = {B'}^{\{1,\ldots,2c_r\}}/{B'}^{\{c_r+1,\ldots,2c_r\}}_{\{c_r+1,\ldots,2c_r\}}$ is real completable to rank $c_r-r$.  Submatrix $C$ has unspecified entry set $G(c_r,1)$, so it is real completable to rank $c_r - r$ by induction hypothesis.

 Therefore $A$ is generically real completable to rank $4c_r-r$, which is corank $r+1$.  See \Cref{figGn1} for an illustration of the proof process.
\end{proof}

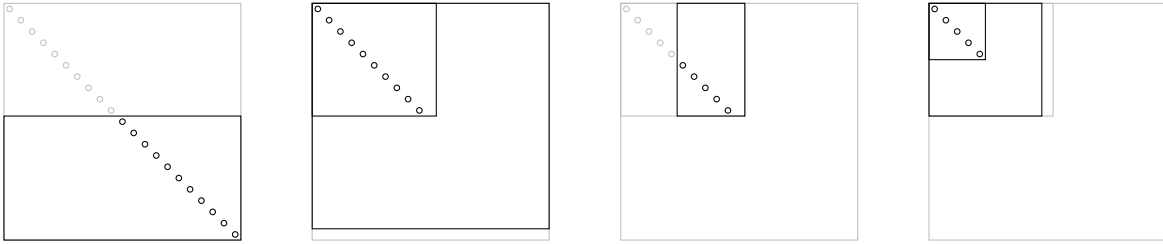
\begin{figure}
  \begin{minipage}{0.24\textwidth}
 \centering
 \begin{tikzpicture}[scale=.15]
\foreach \x in {0,...,9}
  \draw[lightgray] (\x,\x*-1) circle (7pt);
\foreach \x in {10,...,20}
  \draw (\x,\x*-1) circle (7pt);
\draw[lightgray] (-.5,.5) rectangle (20.5,-20.5);
\draw (-.5,-10+.5) rectangle (20.5,-20.5);
\end{tikzpicture}
\end{minipage}
\begin{minipage}{0.24\textwidth}
\centering
 \begin{tikzpicture}[scale=.15]
\foreach \x in {0,...,9}
  \draw (\x,\x*-1) circle (7pt);
\draw[lightgray] (-.5,.5) rectangle (20+.5,-20-.5);
\draw (-.5,.5) rectangle (20+.5,-20+.5);
\draw (-.5,.5) rectangle (10+.5,-10+.5);
\end{tikzpicture}
\end{minipage}
\begin{minipage}{0.24\textwidth}
\centering
 \begin{tikzpicture}[scale=.15]
\foreach \x in {0,...,4}
  \draw[lightgray] (\x,\x*-1) circle (7pt);
\foreach \x in {5,...,9}
  \draw (\x,\x*-1) circle (7pt);
\draw[lightgray] (-.5,.5) rectangle (20+.5,-20-.5);
\draw[lightgray] (-.5,.5) rectangle (10+.5,-10+.5);
\draw (5-.5,.5) rectangle (10+.5,-10+.5);
\end{tikzpicture}
\end{minipage}
\begin{minipage}{0.24\textwidth}
\centering
 \begin{tikzpicture}[scale=.15]
\foreach \x in {0,...,4}
  \draw (\x,\x*-1) circle (7pt);
\draw[lightgray] (-.5,.5) rectangle (20+.5,-20-.5);
\draw[lightgray] (-.5,.5) rectangle (10+.5,-10+.5);
\draw (-.5,.5) rectangle (10-.5,-10+.5);
\draw (-.5,.5) rectangle (5-.5,-5+.5);
\end{tikzpicture}
\end{minipage}
\caption{Pictured is $G(c_{r+1},1)^\comp$ for the case $r = 2$, $c_r = 5$ and $c_{r+1} = 21$.  The submatrix of the last $2c_r+1$ rows can be uniquely completed to rank $2c_r$.  Removing the last row and applying Schur complement, the problem can be reduced to completing the top-left $2c_r\times (2c_r+1)$ submatrix.  The last $c_r+1$ columns can be uniquely completed to rank $c_r$.  Removing the last column and applying Schur complement, the problem can be reduced to completing the top-left $c_r\times c_r$ submatrix.  This submatrix, $G(c_r,1)^\comp$, is completable to corank $r$ so the entire matrix is completable to corank $r+1$.}
\label{figGn1}
\end{figure}

 \Cref{1circulant} recovers \Cref{5diag} for the case $r=2$, by showing that $G(5,1)^\comp$ is generically real completable to corank 2.  The bound $c_2 = 5$ is tight since $G(4,1)^\comp$ is not generically real completable to corank 2.  For $r=3$, we get that $G(n,1)^\comp$ is generically real completable to corank 3 for $n \geq 21$, but this could be far from tight.  The smallest $n$ for which $\gcc(G(n,1)^\comp) = 3$ is 9.  The growth of $c_r$ is $O(4^r)$, while the growth of the smallest $n$ for which $\gcc(G(n,1)^\comp) = r$ is $O(r^2)$.
 
 \smallskip
 For $r > 1$, the completion process in the proof gives a unique completion of $G(c_r,1)^\comp$ to a variety cut out by minors that is strictly contained in the variety of matrices of rank at \mbox{most $r$}.  The generic fiber in the rank variety has positive dimension, for example in the $r=3$ case the dimension is 12.

 We generalize \Cref{1circulant} to the case of $G(n,k)^\comp$ for $k \geq 1$.  In order to do so, for $0 < k \leq n+1$ define 
 \[ \struc{G'(n,k)} = \{(i,j) \in [n]\times [n] \mid i \leq j < i+k\}. \]
 Note that $G'(n,k) \subseteq G(n,k)$.
 \Cref{figG'73} shows the case $n=7$ and $k=3$. 
 
The number of unspecified entries in $G'(n,k)$ is $nk - k(k-1)/2$ and this holds even for $n = k-1$. In the particular case of $n = k-1$ we have the following lemma.
\begin{lemma}
For any $k \geq 1$, $G'(k-1,k)^\comp$ has $\lceil k/2 \rceil$ as its only typical corank.
\end{lemma}
\begin{proof}
The statement follows from the fact that the unspecified entries include a $\lceil k/2 \rceil \times \lceil k/2 \rceil$ block, and the specified entries include a $\lfloor k/2 \rfloor \times \lfloor k/2 \rfloor$ block.
\end{proof}

\begin{figure}
\begin{minipage}{.35\textwidth}
 \centering
\begin{tikzpicture}
\foreach \x in {0,...,6}
  \draw[fill=black] (\x*360/7:.8) circle (3pt);
\foreach \x in {0,...,6}
  \draw[fill=black] (\x*360/7:1.8) circle (3pt);
\foreach \x in {0,...,6}
  \draw (\x*360/7:.8) -- (\x*360/7:1.8);
\foreach \x in {0,...,6}
  \draw (\x*360/7:.8) -- (360/7+\x*360/7:1.8);
\foreach \x in {0,...,6}
  \draw (\x*360/7:.8) -- (-360/7+\x*360/7:1.8);
\end{tikzpicture}
\end{minipage}
\begin{minipage}{.35\textwidth}
\centering
\begin{tikzpicture}
\foreach \x in {0,...,6}
  \draw (\x*0.5,\x*-0.5) circle (3pt);
\foreach \x in {0,...,5}
  \draw (0.5+\x*0.5,\x*-0.5) circle (3pt);
\foreach \x in {0,...,4}
  \draw (1+\x*0.5,\x*-0.5) circle (3pt);
\draw (0,-2.5) circle (3pt);
\draw (0,-3) circle (3pt);
\draw (0.5,-3) circle (3pt);
\draw (-.25,.25) rectangle (3.25,-3.25);
\end{tikzpicture}
\end{minipage}
\caption{Bipartite graph $G(7,3)$ and its matrix entries.}
\label{figG73}
\end{figure}

\begin{figure}[H]
\begin{minipage}{.35\textwidth}
 \centering
\begin{tikzpicture}
\foreach \x in {0,...,6}
\draw[fill=black] (\x*.8,1.5) circle (3pt);
\foreach \x in {0,...,6}
\draw[fill=black] (\x*.8,0) circle (3pt);
\foreach \x in {0,...,6}
\draw (\x*.8,1.5) -- (\x*.8,0);
\foreach \x in {0,...,5}
\draw (\x*.8,0) -- (\x*.8+.8,1.5);
\foreach \x in {0,...,4}
\draw (\x*.8,0) -- (\x*.8+1.6,1.5);
\end{tikzpicture}
\end{minipage}
\begin{minipage}{.35\textwidth}
\centering
\begin{tikzpicture}
\foreach \x in {0,...,6}
  \draw (\x*0.5,\x*-0.5) circle (3pt);
\foreach \x in {0,...,5}
  \draw (0.5+\x*0.5,\x*-0.5) circle (3pt);
\foreach \x in {0,...,4}
  \draw (1+\x*0.5,\x*-0.5) circle (3pt);
\draw (-.25,.25) rectangle (3.25,-3.25);
\end{tikzpicture}
\end{minipage}
\caption{Bipartite graph $G'(7,3)$ and its matrix entries.}
\label{figG'73}
\end{figure}

 \begin{proposition}\label{kcirculant}
 For $k \geq 1$ and $m \geq 0$, $G'(n,k)^\comp$ is generically real completable to corank $\lceil k/2\rceil + mk$ for all $n \geq ((4k-3)4^m - k)/3$.
\end{proposition}

\begin{proof}
 We follow the proof strategy of \Cref{1circulant}
 Let $\struc{c_m} = ((4k-3)4^m - k)/3$.  As above, it suffices to prove the case of $n = c_m$.  Proceed by induction on $m$.  For $m = 0$, $c_0 = k-1$.  As shown above, $G'(k-1,k)$ can be generically real completed to rank $\lceil k/2\rceil$.  For $m \geq 0$, $c_{m+1} = 4c_m + k$.
 
 By a row and column permutation, $G'(4c_m+k,k)$ can be expressed as
  \[ S = G(4c_m+k,k) \setminus \{(i,j) \mid i > c_m \text{ and } j \leq c_m\}, \]
 (see \Cref{figG'nk}).

The induction hypothesis is that $G'(c_m,k)$ is generically real completable to corank $r = \lceil k/2\rceil + mk$.
Let $A$ be a generic partially filled $c_{m+1}\times c_{m+1}$ matrix with unspecified entries $S$.  It follows from \Cref{kcore}, that the submatrix of the last $2c_m+k$ rows is uniquely completable to rank $2c_m$.  Let $A'$ be the partial filling of $A$ with these entries filled.  Removing the last $k$ rows, ${A'}_{\{1,\ldots,4c_m\}}$ is a generic partial filling of an $4c_m\times (4c_m+k)$ matrix.  By \Cref{grow}, ${A'}_{\{1,\ldots,4c_m\}}$ is real completable to rank $4c_m - r$ if and only if $B = {A'}_{\{1,\ldots,4c_m\}}/{A'}^{\{2c_m+1,\ldots,4c_m+k\}}_{\{2c_m+1,\ldots,4c_m\}}$ is real completable to rank $2c_m - r$.
 
 The submatrix $B^{\{c_m+1,\ldots,2c_m+k\}}$ is uniquely completable to rank $c_m$ by \Cref{kcore}, and let $B'$ be the partial completion of $B$ with these entries filled.  Again by \Cref{grow}, ${B'}^{\{1,\ldots,2c_m\}}$ is real completable to rank $2c_m-r$ if and only if the Schur complement $C = {B'}^{\{1,\ldots,2c_m\}}/{B'}^{\{c_m+1,\ldots,2c_m\}}_{\{c_m+1,\ldots,2c_m\}}$ is real completable to rank $c_m-r$.  Submatrix $C$ has unspecified entry set $G'(c_m,k)$, so it is real completable to rank $c_m - r$.

 Therefore $A$ is generically real completable to rank $4c_m-r$, which is corank $r+k$.  See \Cref{figG'nk} for an illustration of the proof process.
\end{proof}

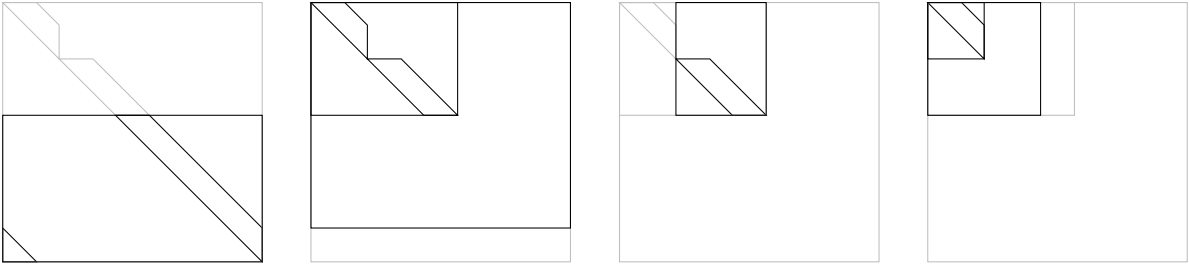
\begin{figure}
  \begin{minipage}{0.24\textwidth}
 \centering
 \begin{tikzpicture}[scale=.15]
\draw[lightgray] (0,0) -- (3,0) -- (5,-2) -- (5,-5) -- cycle;
\draw[lightgray] (5,-5) -- (8,-5) -- (13,-10) -- (10,-10) -- cycle;
\draw (10,-10) -- (13,-10) -- (23,-20) -- (23,-23) -- cycle;
\draw (0,-20) -- (3,-23) -- (0,-23) -- cycle;
\draw[lightgray] (0,0) rectangle (23,-23);
\draw (0,-10) rectangle (23,-23);
\end{tikzpicture}
\end{minipage}
\begin{minipage}{0.24\textwidth}
\centering
 \begin{tikzpicture}[scale=.15]
\draw (0,0) -- (3,0) -- (5,-2) -- (5,-5) -- cycle;
\draw (5,-5) -- (8,-5) -- (13,-10) -- (10,-10) -- cycle;
\draw[lightgray] (0,0) rectangle (23,-23);
\draw (0,0) rectangle (23,-20);
\draw (0,0) rectangle (13,-10);
\end{tikzpicture}
\end{minipage}
\begin{minipage}{0.24\textwidth}
\centering
 \begin{tikzpicture}[scale=.15]
\draw[lightgray] (0,0) -- (3,0) -- (5,-2) -- (5,-5) -- cycle;
\draw (5,-5) -- (8,-5) -- (13,-10) -- (10,-10) -- cycle;
\draw[lightgray] (0,0) rectangle (23,-23);
\draw[lightgray] (0,0) rectangle (13,-10);
\draw (5,0) rectangle (13,-10);
\end{tikzpicture}
\end{minipage}
\begin{minipage}{0.24\textwidth}
\centering
 \begin{tikzpicture}[scale=.15]
\draw (0,0) -- (3,0) -- (5,-2) -- (5,-5) -- cycle;
\draw[lightgray] (0,0) rectangle (23,-23);
\draw[lightgray] (0,0) rectangle (13,-10);
\draw (0,0) rectangle (10,-10);
\draw (0,0) rectangle (5,-5);
\end{tikzpicture}
\end{minipage}
\caption{The proof of \Cref{kcirculant} for $G'(n,k)$ follows that of \Cref{1circulant} for $G(n,1)$, except that $k$ rows, and then $k$ columns are eliminated in the process instead of 1.  The polygonal areas represent the unspecified parts of the matrices.}
\label{figG'nk}
\end{figure}

\begin{corollary}\label{triangle}
 For any $r \geq 0$ and $k \geq 1$, $G'(n,k)^\comp$ and $G(n,k)^\comp$ are generically real completable to corank $r$ for $n$ sufficiently large.
\end{corollary}

In the search for examples with many typical coranks, \cite{BBS} suggests that perhaps for some fixed $k$, $G(n,k)$ might have the same minimum typical corank for all $n \gg 0$, even as the generic completion coranks grow.  However, the above corollary shows that this is not the case.  Still, the lower bound we get for the minimum typical rank grows slower in $n$ than the generic completion corank, so it does not rule out increasingly many typical coranks.

\smallskip

For the specific case of $n=6$ and $k=2$ we get the following typical corank.
 \begin{proposition}\label{G64e}
  $\tc((G(6,2))^\comp) = \tc((G'(6,2))^\comp) = \{3\}$.
 \end{proposition}
 \begin{proof}
  By \Cref{kcirculant}, $G'(6,2)^\comp$ is generically real completable to corank 3, and then so is $G(6,2)^\comp$.  Since $|G'(6,2)| = 11 < 4^2$, $G'(6,2)^\comp$ cannot be generically completed to corank 4.  Similarly for $G(6,2)^\comp$.
 \end{proof}

More generally, $G(n,2)^\comp$ does not have typical corank 2 for any $n \geq 6$.  This gives a negative answer to Question 5.5 of \cite{BBS}, which asks if 2 is a typical corank of $G(8,2)^\comp$.  It is easy to check that $\tc(G(4,2)^\comp) = \{2\}$.  In between, $G(5,2)^\comp$ has generic completion corank 3, but whether it can be generically completed to corank 3 over the reals remains open.
 \begin{question}
  What is $\tc(G(5,2)^\comp)$?
 \end{question}

A first step towards answering this question may be to take a different point of view. This problem translates nicely to the subsequent question rooted in Schubert calculus: Given 5 planes in $\P^2$ in general position, is there a plane that intersects all of them? 
If the answer to that question would be affirmative, then it is possible to complete $G(5,2)^\comp$ generically to rank $2$ over the reals, i.e., $\tc(G(5,2)^\comp)=\{3\}$.

\medskip
Following the result that $\tc(G(4,1)^\comp) = \{1,2\}$, a question is whether there are examples with typical coranks $\{n,n+1\}$ for any $n \geq 1$, or typical ranks $\{n,n+1\}$ for any $n\geq 2$.  The answer to both is yes.  The first answer follows immediately from \Cref{grow}.  By keeping the unspecified entries as $G(4,1)$ and growing the size of the matrix, the typical coranks remain $\{1,2\}$.

A family of examples answering the second question is also easy to construct.  For $n \geq 1$, let the unspecified entries of an $(n+3) \times (n+3)$ matrix be
  \[ U =  (\{3,\ldots,n+3\}\times \{3,\ldots,n+3\} \setminus \{3,4\} \times \{3,4\}) \cup G(4,1). \]
If a generic partial filling has the top left $4 \times 4$ submatrix real completable to rank 2, then the remaining rows can be uniquely completed to be in the row span, and similarly for the remaining columns.  On the other hand, if the $4\times 4$ submatrix can only be real completed to rank 3, then the matrix cannot be completed to rank 2, so the 3 is also a typical rank.  See \Cref{fig2typranks} for an example of each with $7\times 7$ matrices.

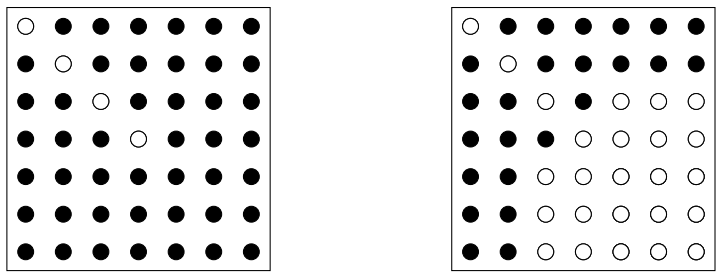
\begin{figure}
\begin{minipage}{.35\textwidth}
 \centering
\begin{tikzpicture}
\foreach \x in {0,...,6}
  \foreach \y in {0,...,6}
    \draw[fill=black] (\x*0.5,\y*-0.5) circle (3pt);
\foreach \x in {0,...,3}
  \draw[fill=white] (\x*0.5,\x*-0.5) circle (3pt);
\draw (-.25,.25) rectangle (3.25,-3.25);
\end{tikzpicture}
\end{minipage}
\begin{minipage}{.35\textwidth}
\centering
\begin{tikzpicture}
\foreach \x in {0,...,6}
  \foreach \y in {0,...,6}
    \draw[fill=black] (\x*0.5,\y*-0.5) circle (3pt);
\foreach \x in {2,...,6}
  \foreach \y in {4,...,6}
    \draw[fill=white] (\x*0.5,\y*-0.5) circle (3pt);
\foreach \x in {4,...,6}
  \foreach \y in {2,...,6}
    \draw[fill=white] (\x*0.5,\y*-0.5) circle (3pt);
\foreach \x in {0,...,3}
  \draw[fill=white] (\x*0.5,\x*-0.5) circle (3pt);
\draw (-.25,.25) rectangle (3.25,-3.25);
\end{tikzpicture}
\end{minipage}
\caption{The left matrix has typical coranks $\{1,2\}$.  The right matrix has typical ranks $\{2,3\}$. These typical coranks and ranks respectively remain constant for the generalizations to any size matrices.  Black and white dots represent specified and unspecified entries respectively.}
\label{fig2typranks}
\end{figure}

 \smallskip
We prove however that $G'(5,2)^\comp$ exhibits both 2 and 3 as typical coranks. 
 \begin{proposition}\label{G'52}
  $\tc(G'(5,2)^\comp) = \{2,3\}$.
 \end{proposition}
 \begin{proof}
  Let $A$ be the generic partial filled $5\times 5$ matrix with unspecified entries $G'(5,2)$. To complete $A$ to rank 2, all $3\times 3$ minors must vanish.  Since $A_{\{3,4,5\}}^{\{1,2,3\}}$ has only one unspecified entry, the minor condition uniquely specifies entry $(3,3)$.  Let $A'$ be the partial filling with $(3,3)$ specified. The $4\times 4$ submatrix $B = {A'}_{\{1,2,3,5\}}^{\{1,3,4,5\}}$ has four unspecified entries along the diagonal.  Because the choice of entry $(3,3)$ depends on values that do not appear in $B$, $B$ is generic.  By Example 3.1 of \cite{BBS}, $B$ has typical ranks 2 and 3.
  
  Suppose that $B$ is completed to rank 2.  Then the remaining two unspecified entries of ${A'}_{\{1,2,3,5\}}$ can be chosen so that is has rank 2, and similarly for ${A'}^{\{1,3,4,5\}}$.  Since $\det {A'}_{\{3,4,5\}}^{\{1,2,3\}}$ vanishes, the resulting completion has rank 2.  If instead $B$ is completed to rank 3, $A$ can only be completed to rank 3.
 \end{proof}
Using the same argument, \Cref{G'52} can be generalized to show that for $\struc{H} = G(n+3,n) \setminus E$, where $\struc{E}$ is an appropriately chosen set of $n-1$ entries, $\tc(H^\comp) = \{n,n+1\}$ for all $n \geq 3$.  This gives another family of examples with increasing pairs of typical coranks.

\section{Discussion}\label{Sec:Discussion}

 Every example so far of an entry set with more than one typical corank has the property that for a generic real partial filling, the number of complex completions to some typical corank $r$ is finite, or at least one of the entries has only a finite number of values in the set of completions.  We have very few examples to draw a pattern from, but it still raises the question of whether this always holds.
 
 \begin{question}
  Is there an unspecified entry set $U \subseteq [n]\times [n]$ with the property that for a generic real partial filling, every entry in $U$ has an infinite number of possible complex values in the set of completions to corank $r$, but there are no real completions?
 \end{question}
 
 If the answer is no, then three or more typical coranks is impossible.

\subsection*{Acknowledgments}
The beginnings of this project were facilitated by the NSF-supported (DMS-1439786) Institute for Computational and Experimental Research in Mathematics (ICERM), where both authors were in residence during the semester program on Nonlinear Algebra in the Fall of 2018.  We thank Greg Blekherman and Daniel Bernstein for introducing us to the topic and for their contributions to early stages of the project, and Greg's consultation on the paper.  We also acknowledge the help of Rainer Sinn.

\bibliographystyle{amsalpha}
\bibliography{bib}

\end{document}